\documentclass[12pt]{amsart}

\usepackage{amsmath,amssymb,mathrsfs}

\usepackage{amsthm} 

\usepackage{srcltx} 
\usepackage{times}
\usepackage{amsmath,amssymb,mathrsfs}

\usepackage{amsthm} 
\usepackage{srcltx}

 \usepackage[dvips]{graphicx,color}

\usepackage{amsopn,cite,mathrsfs}
\usepackage{amsfonts}
\usepackage{eucal}

\numberwithin{equation}{section}

\newtheorem{thm}{Theorem}[section]
\newtheorem{lemma}[thm]{Lemma}

\newtheorem*{proof-mthm}{Proof of Main Theorem}

 \newcommand{\blem}{\begin{lemma}}
 \newcommand{\elem}{\end{lemma}}
\newcommand{\bpve}{\begin{proof}}
\newcommand{\epve}{\end{proof}}

\newtheorem*{thmA}{Theorem}

\usepackage{amssymb}
\usepackage{amsmath}
\usepackage{amsfonts}
\usepackage[colorlinks=true,linkcolor=red]{hyperref}
\theoremstyle{plain}

\numberwithin{equation}{section}

\begin{document}

\title{On the Commutativity of the Berezin Transform}

\author{Alexander Borichev}

\address{Aix-Marseille University, CNRS, I2M, Marseille,
France}

\email{alexander.borichev@math.cnrs.fr}

\author{G\'erard Fantolini}

\address{Aix-Marseille University, CNRS, I2M, Marseille,
France.}

\email{fantolinigerard@gmail.com}

\author{El-Hassan Youssfi}

\address{Aix-Marseille University, CNRS, I2M, Marseille,
France.}

\email{el-hassan.youssfi@univ-amu.fr}


\begin{abstract}
We consider the commutativity problem for the Berezin transform on weighted Fock spaces. 
Given a real number $m>0$, for every $\alpha >0$ we denote by $B_{\alpha}$ the Berezin transform associated  to the measure  $\mu_{m}^{\alpha}$ 
with density proportional to $e^{-\alpha |z|^m}$ with respect to Lebesgue measure on the complex plane and normalized so that $\mu_{\phi}^{\alpha}(\mathbb C) =1$. 
We show that the commutativity relation $B_{\alpha}B_{\beta}f=B_{\beta}B_{\alpha}f$ holds for all $f\in L^{\infty}(\mathbb C)$ and $\alpha,\beta> 0$  if and only if $m=2$.  \end{abstract}

\maketitle

\section{Introduction and statement of the main result}

We fix a real number $m>0$ and for every $\alpha >0$ we consider the (normalized) measure on the complex plane  
$$
d\mu_{\alpha,m}(z)= \frac{m\alpha^{2/m}}{2\pi\Gamma(2/m)}e^{-\alpha |z|^m}\,dA(z),
$$ 
where  $dA$ is Euclidean area measure on the complex plane $\mathbb C$ and $\Gamma$ is the Euler gamma function.
We denote by $L_{\alpha,m}^{2}(\mathbb C)$ the space of all square integrable functions with respect to the measure $\mu_{\alpha,m}$ and let 
$$
F_{\alpha,m}^{2}(\mathbb C)=L_{\alpha,m}^{2}(\mathbb C)\cap {\rm Hol}(\mathbb C)
$$
denote the corresponding weighted Fock space. The space $F_{\alpha,m}^{2}(\mathbb C)$ is a Hilbert space of entire functions, and we denote by $K_{\alpha,m}(\cdot,\cdot)$ its reproducing kernel,
$$
f(z)=\langle f,K_{\alpha,m}(\cdot,z) \rangle,\qquad f\in F_{\alpha,m}^{2}(\mathbb C),\, z\in \mathbb C.
$$ 
Let $P_+$ be the orthogonal projection from $L_{\alpha,m}^{2}(\mathbb C)$ onto $F_{\alpha,m}^{2}(\mathbb C)$. Given $f\in L^{\infty}(\mathbb C)$, we define the Toeplitz operator $T_f$ on $F_{\alpha,m}^{2}(\mathbb C)$ by the formula 
$$
T_fg=P_+(fg),\qquad g\in F_{\alpha,m}^{2}(\mathbb C).
$$
The Berezin transform of $T_f$ (or the Berezin transform of $f$) is given by
\begin{multline*}
(B_{\alpha,m}T_f)(z)=(B_{\alpha,m}f)(z)=\frac{\langle T_fK_{\alpha,m}(\cdot,z),K_{\alpha,m}(\cdot,z)\rangle}{K_{\alpha,m}(z,z)}
\\=\displaystyle\int_{\mathbb C}f(w) \frac{|K_{\alpha,m}(w,z)|^2}{K_{\alpha,m}(z,z)}\,d\mu_{m}^{\alpha}(w), \qquad z\in\mathbb C.
\end{multline*}
For the Berezin transform of operators in a more general context and for the Berezin transform of functions see \cite{B,E1,E2,Z}.

Since $\|K_{\alpha,m}(\cdot,z) \|^2=K_{\alpha,m}(z,z)$, it is clear that $B_{\alpha,m}$ is a contraction on $L^{\infty}(\mathbb C) $ 
for every $\alpha>0$.

In the case of the classical Fock spaces, that is, when $m=2$, it is known (see, for instance, \cite[Chapter 2]{Z} that the Berezin transform satisfies the commutativity relation $B_{\alpha,2}B_{\beta,2}=B_{\beta,2}B_{\alpha,2}$ on 
$L^{\infty}(\mathbb C)$.
 

The problem we consider here is to find out for which positive real numbers $m$ does the latter commutativity remain true.  The main goal of this paper is to establish the following:

\begin{thmA}\label{main}  Assume that $m>0$  and consider the test functions  $f_\delta(z):=e^{-\delta|z|^m}$, $z\in\mathbb C$, $\delta>0$. Then the following are equivalent
\begin{itemize}
\item[(1)] The commutativity  relation   $B_{\alpha,m}B_{\beta,m}f=B_{\beta,m}B_{\alpha,m}f$  holds for all $f\in L^{\infty}(\mathbb C)$ and  all $\alpha, \beta > 0$.
\item[(2)] The commutativity  relation   $B_{\alpha,m}B_{\beta,m}f_\delta =B_{\beta,m}B_{\alpha,m}f_\delta$  holds for all  $\alpha, \beta, \delta >0.$
\item[(3)] The commutativity  relation   $(B_{\alpha,m}B_{\beta,m}f_\delta)(0) =(B_{\beta,m}B_{\alpha,m}f_\delta)(0)$  holds for all $\alpha, \beta, \delta >0.$
\item[(4)] $m=2.$
\end{itemize}
\end{thmA} 

The proof of this result relies on expansions of the reproducing kernel using the sequence of the Stieltjes even moments $(s_d)$ of the measure $ e^{-\alpha t^{m}}\,dt$ on $[0,+\infty[$ which will be defined in the sequel. For similar arguments involving the Stieltjes moments see \cite{BY,CF,SY}.

\section{ Proof of the  main result}

We start with the equality
$$
\|z^n\|^2_{\alpha,m}=\frac{\Gamma(2(n+1)/m)}{\Gamma(2/m)\alpha^{2n/m}},\qquad n\in \mathbb Z_+.
$$
Therefore, we have 
$$ 
K_{\alpha,m}(z,w)=\sum_{n\ge 0}\frac{(z\overline{w})^n}{\|z^n\|^2_{\alpha,m}}=
\Gamma(2/m)
\sum_{n\ge 0}\frac{(\alpha^{\frac{2}{m}}z\overline{w})^{n}}{\Gamma(2(n+1)/m)}.
$$

Next we define the Stieltjes moments
$$
s_n(\alpha,m):= \alpha^{-2n/m}\Gamma(2(n+1)/m),
$$
and consider an entire function
$$
S_{\alpha,m}(\zeta):= \sum_{n\ge0}\frac{\zeta^n}{s_n(\alpha,m)}. 
$$
Then
$$
K_{\alpha,m}(z,w)=\Gamma(2/m)S_{\alpha,m}(z\overline{w}).
$$

Consequently, the Berezin transform can be expressed as 
\begin{multline*}
(B_{\alpha,m}f)(z)=\int_{\mathbb C}f(w)\frac{|K_{\alpha,m}(w,z)|^2}{K_{\alpha,m}(z,z)}\,d\mu_{\alpha,m}(w)\\
 = \frac{m\alpha^{2/m}}{2\pi S_{\alpha,m}(|z|^2)}  \int_{\mathbb C}f(w) |S_{\alpha,m}(z\overline{w})|^2
 e^{-\alpha |w|^m}\,dA(w).
\end{multline*}

Since 
$$
S_{\alpha,m}(0)=s_0(\alpha,m)^{-1}= \Gamma(2/m)^{-1},
$$
we have
\begin{equation}
(B_{\alpha,m}f)(0)=\frac{m\alpha^{2/m}}{2\pi\Gamma(2/m)}  \int_{\mathbb C}f(w) e^{-\alpha |w|^m}\,dA(w).
\label{e0}
\end{equation}

Furthermore, 
$$
|S_{\alpha,m}(\zeta)|^2=\sum_{k\ge 0}\sum_{0\le d\le k}\frac{\zeta^d\overline{\zeta}^{k-d}}{s_d(\alpha,m)s_{k-d}(\alpha,m)},\qquad \zeta\in\mathbb C.
$$
Applying the Berezin transform to the test function $f_{\delta}(w)=e^{-\delta|w|^{m}}$, $\delta\ge 0$, and using pairwise orthogonality of $w^q$, $q\in\mathbb Z_+$, on the circles $\partial D(0,R)$, $R>0$, we obtain 
\begin{multline}
(B_{\alpha,m}f_\delta)(z)\\= 
\frac{m\alpha^{2/m}}{2\pi S_{\alpha,m}(|z|^2)}  \sum_{k\ge 0} 
\sum_{0\le d\le k}\int_{\mathbb C}f_\delta(w)  
\frac{(z\overline{w})^d(w\overline{z})^{k-d}}{s_{d}(\alpha,m)s_{k-d}(\alpha,m)} e^{-\alpha |w|^m}\,dA(w) \\
 =   \frac{m\alpha^{2/m}}{2\pi S_{\alpha,m}(|z|^2)}\sum_{n\ge 0} |z|^{2n}  \int_{\mathbb C}
\frac{ |w|^{2n} }{s_n^2(\alpha,m)}     e^{-(\alpha+\delta) |w|^m}\,dA(w).\label{e1}
\end{multline}

By \eqref{e0} and \eqref{e1} we have 
\begin{multline*}
(B_{\beta,m}B_{\alpha,m}f_\delta)(0) = \frac{m\beta^{2/m}}{2\pi\Gamma(2/m)}   \int_{\mathbb C}  (B_{\alpha,m}f_\delta)(z) e^{-\beta |z|^m}\,dA(z)\\
=  \frac{m^2 (\alpha\beta)^{2/m} }{4\pi^2\Gamma(2/m)}   \sum_{n\ge 0}   \int_{\mathbb C} \frac{|z|^{2n}}{S_{\alpha,m}(|z|^2)}e^{-\beta |z|^m}\,dA(z)\,\int_{\mathbb C}  
\frac{ |w|^{2n}}{s_n^2(\alpha,m)} e^{-(\alpha + \delta)|w|^m}  \,dA(w) \\
= \frac{m (\alpha\beta)^{2/m} }{2\pi\Gamma(2/m)}  \sum_{n\ge 0} \frac1{s_n^2(\alpha,m)}   \left(\frac{1}{\delta+\alpha}\right)^{\frac{2n+2}{m}}\Gamma\left(\frac{2n+2}{m} \right)\int_{\mathbb C}  \frac{  |z|^{2n} e^{-  \beta |z|^m}  }{ S_{\alpha,m}(|z|^2)}  \,  dA(z)\\
 =    \frac{m (\alpha\beta)^{2/m} }{2\pi\Gamma(2/m)} \sum_{n\ge 0}  
  \frac{ \alpha^{4n/m}}{\Gamma\left(\frac{2n+2}{m} \right)}   \left(\frac{1}{\delta+\alpha}\right)^{\frac{2n+2}{m}}\int_{\mathbb C}  \frac{  |z|^{2n} e^{-  \beta |z|^m}  }{ S_{\alpha,m}(|z|^2)}  \,dA(z).
\end{multline*}
(Here we use that $f_\delta$, $B_{\alpha,m}f_\delta$ are positive and bounded so that one can apply the Foubini theorem.)

Thus,
\begin{multline}\label{BBR1}
(B_{\beta,m}B_{\alpha,m}f_\delta)(0) \\=  \frac{m (\alpha\beta)^{2/m}}{2\pi\Gamma(2/m)}\sum_{n\ge 0}  
   \left(\frac{1}{\delta+\alpha}\right)^{(2n+2)/m} \frac{ \alpha^{4n/m}}{\Gamma(\frac{2n+2}{m})} \int_{\mathbb C}  \frac{  |z|^{2n} e^{-  \beta |z|^m}  }{ S_{\alpha,m}(|z|^2)}\,dA(z).
\end{multline}

Set
$$
U_{\alpha,\beta}(n)=\frac{ \alpha^{4n/m}}{\Gamma(\frac{2n+2}{m})} \int_{\mathbb C}  \frac{  |z|^{2n} e^{-  \beta |z|^m}  }{ S_{\alpha,m}(|z|^2)}\,dA(z).
$$
Since the series in \eqref{BBR1} converges for $\delta=0$, we have
\begin{equation}
U_{\alpha,\beta}(n)\le C\alpha^{2n/m},\qquad  n\ge 0.
\label{estar}
\end{equation}

Now we rewrite \eqref{BBR1} as
\begin{equation}
(B_{\beta,m}B_{\alpha,m}f_\delta)(0) \\=  \frac{m (\alpha\beta)^{2/m}}{2\pi\Gamma(2/m)}\sum_{n\ge 0}  
   \left(\frac{1}{\delta+\alpha}\right)^{(2n+2)/m} U_{\alpha,\beta}(n).
\label{e7}
\end{equation}

\begin{lemma}\label{l1} Let $m,\alpha,\beta>0$. If the commutativity relation 
$$ 
(B_{\beta,m}B_{\alpha,m}f_\delta)(0)=(B_{\alpha,m} B_{\beta,m}f_\delta)(0)
$$
holds for every $\delta>0$, then 
\begin{equation}
U_{\alpha,\beta}(n)=U_{\beta,\alpha}(n),\qquad 0\le n<\frac{m}2.
\label{e8}
\end{equation}
\label{lem1}
\end{lemma}

\begin{proof} 
Suppose that for every $\delta>0$ we have 
$(B_{\beta,m}B_{\alpha,m}f_\delta)(0)=(B_{\alpha,m} B_{\beta,m}f_\delta)(0)$. Then by \eqref{e7} we obtain 
\begin{equation}
\sum_{n\ge 0} \left(\frac{1}{\delta+\alpha}\right)^{\frac{2n+2}m}U_{\alpha,\beta}(n)
\\=\sum_{n\ge 0} \left(\frac{1}{\delta+\beta}\right)^{\frac{2n+2}m}U_{\beta,\alpha}(n),\quad \delta>0.\label{e3}
\end{equation}
We set $t=\delta^{-2/m}$ and rewrite \eqref{e3} as 
\begin{multline}
\sum_{n\ge 0} \frac{t^{n+1}}{(1+\alpha t^{m/2})^{(2n+2)/m}}U_{\alpha,\beta}(n)\\=\sum_{n\ge 0} \frac{t^{n+1}}{(1+\beta t^{m/2})^{(2n+2)/m}}U_{\beta,\alpha}(n),\qquad t>0.
\label{e4}
\end{multline}
Then, using \eqref{estar}, we get
$$
\sum_{n\ge 0}t^n(U_{\alpha,\beta}(n)-U_{\beta,\alpha}(n))=O(t^{m/2}), \qquad t\to 0,
$$
and hence, 
$$
U_{\alpha,\beta}(n)=U_{\beta,\alpha}(n),\qquad 0\le n<\frac{m}2.
$$
\end{proof}

\begin{lemma}\label{mnoneven} Let $m\not\in 2\mathbb Z_+$. If $\alpha\not=\beta$, then the commutativity relation 
\begin{equation}
(B_{\beta,m}B_{\alpha,m}f_\delta)(0)=(B_{\alpha,m} B_{\beta,m}f_\delta)(0),  \quad \delta>0,
\label{estar2}
\end{equation}
does not hold.
\end{lemma}

\begin{proof} 
By \eqref{e4} and \eqref{e8}, relation \eqref{estar2} implies that 
\begin{multline*}
\Bigl(\frac{1}{(1+\alpha t^{m/2})^{2/m}}-\frac{1}{(1+\beta t^{m/2})^{2/m}}\Bigr)U_{\alpha,\beta}(0)\\=
\sum_{0<n<\frac{m}2} \Bigl(\frac{t^n}{(1+\beta t^{m/2})^{(2n+2)/m}}-\frac{t^n}{(1+\alpha t^{m/2})^{(2n+2)/m}}\Bigr)U_{\alpha,\beta}(n)\\+
\sum_{n>\frac{m}2}  \Bigl(\frac{t^n}{(1+\beta t^{m/2})^{(2n+2)/m}}U_{\beta,\alpha}(n)-\frac{t^n}{(1+\alpha t^{m/2})^{(2n+2)/m}}U_{\alpha,\beta}(n)\Bigr)\\=o(t^{m/2}),\qquad t\to 0,
\end{multline*}
and hence,
$$
(\beta-\alpha)U_{\alpha,\beta}(0)=o(1),\qquad t\to 0,
$$
which is impossible unless $\alpha=\beta$.
\end{proof}

\begin{lemma}\label{moments} Let $m$ be an even positive integer. Suppose that the commutativity relation $(B_{\beta,m}B_{\alpha,m}f_\delta)(0)=(B_{\alpha,m}B_{\beta,m}f_\delta)(0)$
holds for every $\delta>0$ and every $\alpha,\beta>0$. Then $m=2$. 
\end{lemma}

\begin{proof} Let $m=2p$. By \eqref{e4} we have
$$
\sum_{n\ge 0} \frac{t^n}{(1+\alpha t^p)^{(n+1)/p}}U_{\alpha,\beta}(n)=\sum_{n\ge 0} \frac{t^n}{(1+\beta t^p)^{(n+1)/p}}U_{\beta,\alpha}(n),\qquad t>0.
$$

By the binomial formula, 
$$
(1+x)^q=\sum_{s\ge 0}c_{q,s}x^s,\qquad 0\le x<1,
$$
where 
$$
c_{q,0}=1,\quad c_{q,1}=q,\quad c_{q,2}=q(q-1)/2,
$$
we obtain
\begin{multline}
\sum_{n,s\ge 0}t^{n+sp}\alpha^s c_{-(n+1)/p,s}U_{\alpha,\beta}(n)\\=\sum_{n,s\ge 0}t^{n+sp}\beta^s c_{-(n+1)/p,s}U_{\beta,\alpha}(n),\qquad 0<t<1,
\end{multline}
with both series converging absolutely.

Hence, for every $b\in\mathbb Z_+$ we have
$$
\sum_{n+sp=b} \alpha^s c_{-(n+1)/p,s}U_{\alpha,\beta}(n)=\sum_{n+sp=b}\beta^s c_{-(n+1)/p,s}U_{\beta,\alpha}(n).
$$
By Lemma~\ref{l1} we know already that 
\begin{equation}
U_{\alpha,\beta}(0)=U_{\beta,\alpha}(0),\label{tt1}.
\end{equation}

Taking the values $b=p,2p$ we obtain 
\begin{equation}
U_{\alpha,\beta}(p)-\frac{\alpha}p U_{\alpha,\beta}(0)=U_{\beta,\alpha}(p)-\frac{\beta}p U_{\beta,\alpha}(0),\label{tt2}
\end{equation}
and
\begin{multline}
U_{\alpha,\beta}(2p)-\frac{\alpha(p+1)}p U_{\alpha,\beta}(p)+\frac{\alpha^2(p+1)}{2p^2}U_{\alpha,\beta}(0)\\=U_{\beta,\alpha}(2p)-\frac{\beta(p+1)}p U_{\beta,\alpha}(p)+\frac{\beta^2(p+1)}{2p^2}U_{\beta,\alpha}(0).\label{tt3}
\end{multline}

By \eqref{tt1} and \eqref{tt2} we obtain
\begin{equation}
U_{\alpha,\beta}(p)-U_{\beta,\alpha}(p)= \frac{\alpha-\beta}p U_{\alpha,\beta}(0).
\label{tt4}
\end{equation}

By \eqref{tt1}--\eqref{tt3} we obtain
\begin{multline}
U_{\alpha,\beta}(2p)-U_{\beta,\alpha}(2p)\\=\frac{(\alpha-\beta)(p+1)}p U_{\alpha,\beta}(p)-\frac{(\alpha-\beta)^2(p+1)}{2p^2}U_{\alpha,\beta}(0).
\label{tt5}
\end{multline}

Next we use that 
\begin{multline}
\frac{\partial}{\partial \beta}U_{\alpha,\beta}(n)=\frac{\partial}{\partial \beta}\frac{ \alpha^{2n/p}}{\Gamma(\frac{n+1}{p})} \int_{\mathbb C}  \frac{  |z|^{2n} e^{-  \beta |z|^{2p}}  }{ S_{\alpha,2p}(|z|^2)}\,dA(z)\\=
-\frac{ \alpha^{2n/p}}{\Gamma(\frac{n+1}{p})} \int_{\mathbb C}  \frac{  |z|^{2n+2p} e^{-  \beta |z|^{2p}}  }{ S_{\alpha,2p}(|z|^2)}\,dA(z)\\=-\frac{n+1}{\alpha^2p}\cdot\frac{ \alpha^{(2n+2p)/p}}{\Gamma(\frac{n+p+1}{p})} \int_{\mathbb C}  \frac{  |z|^{2n+2p} e^{-  \beta |z|^{2p}}  }{ S_{\alpha,2p}(|z|^2)}\,dA(z)\\=-\frac{n+1}{\alpha^2p}U_{\alpha,\beta}(n+p),\qquad n\ge 0.
\label{tt6}
\end{multline}
By analogy,
\begin{equation}
\frac{\partial}{\partial \alpha}U_{\beta,\alpha}(n)=-\frac{n+1}{\beta^2p}U_{\beta,\alpha}(n+p),\qquad n\ge 0.
\label{tt7}
\end{equation}

As a consequence, we obtain 
\begin{equation}
\left\{
\begin{aligned}
\frac{\partial}{\partial \alpha}U_{\beta,\alpha}(0)&=-\frac{1}{\beta^2p}U_{\beta,\alpha}(p),\\
\frac{\partial}{\partial \beta}U_{\alpha,\beta}(0)&=-\frac{1}{\alpha^2p}U_{\alpha,\beta}(p),\\
\frac{\partial}{\partial \alpha}U_{\beta,\alpha}(p)&=-\frac{p+1}{\beta^2p}U_{\beta,\alpha}(2p),\\ 
\frac{\partial}{\partial \beta}U_{\alpha,\beta}(p)&=-\frac{p+1}{\alpha^2p}U_{\alpha,\beta}(2p).
\end{aligned}
\right.
\label{start}
\end{equation}

Now, let us evaluate 
$$
H=\frac{\partial^2}{\partial \alpha\partial \beta}(\alpha^2\beta^2 U_{\alpha,\beta}(0)).
$$
First, by \eqref{tt1}, \eqref{tt4}, and \eqref{start} we have
\begin{multline}
H=\frac{\partial}{\partial \alpha}\Bigl[ 2\alpha^2\beta U_{\alpha,\beta}(0) -\frac{\beta^2}p U_{\alpha,\beta}(p)\Bigr]\\=
\frac{\partial}{\partial \alpha}\Bigl[ 2\alpha^2\beta U_{\beta,\alpha}(0) -\frac{\beta^2}p U_{\beta,\alpha}(p)-\frac{(\alpha-\beta)\beta^2}{p^2}U_{\beta,\alpha}(0)\Bigr]\\=
4\alpha\beta U_{\beta,\alpha}(0)-\frac{2\alpha^2}{\beta p}U_{\beta,\alpha}(p)+\frac{p+1}{p^2} U_{\beta,\alpha}(2p)-\frac{\beta^2}{p^2}U_{\beta,\alpha}(0)+\frac{\alpha-\beta}{p^3}U_{\beta,\alpha}(p)\label{tttdif}.
\end{multline}
By \eqref{tt1}, $H$ is stable under permutation of $\alpha$ and $\beta$, 
$$
H=\frac{\partial^2}{\partial \beta\partial \alpha}(\beta^2\alpha^2 U_{\beta,\alpha}(0)),
$$
and, hence,
\begin{multline*}
4\alpha\beta U_{\beta,\alpha}(0)-\frac{2\alpha^2}{\beta p}U_{\beta,\alpha}(p)+\frac{p+1}{p^2} U_{\beta,\alpha}(2p)-\frac{\beta^2}{p^2}U_{\beta,\alpha}(0)+\frac{\alpha-\beta}{p^3}U_{\beta,\alpha}(p)\\=
4\alpha\beta U_{\alpha,\beta}(0)-\frac{2\beta^2}{\alpha p}U_{\alpha,\beta}(p)+\frac{p+1}{p^2} U_{\alpha,\beta}(2p)-\frac{\alpha^2}{p^2}U_{\alpha,\beta}(0)+\frac{\beta-\alpha}{p^3}U_{\alpha,\beta}(p).
\end{multline*}

We obtain that
\begin{multline*}
\frac{p+1}{p^2} (U_{\alpha,\beta}(2p)-U_{\beta,\alpha}(2p))+ \frac{\beta^2}{p^2}U_{\beta,\alpha}(0)-\frac{\alpha^2}{p^2}U_{\alpha,\beta}(0)\\
+\frac{2\alpha^2}{\beta p}U_{\beta,\alpha}(p)-\frac{2\beta^2}{\alpha p}U_{\alpha,\beta}(p)
+\frac{\beta-\alpha}{p^3}U_{\alpha,\beta}(p)
+\frac{\beta-\alpha}{p^3}U_{\beta,\alpha}(p)=0,
\end{multline*}
and by \eqref{tt1} and \eqref{tt5} we get
\begin{multline*}
\frac{(\alpha-\beta)(p+1)^2}{p^3} U_{\alpha,\beta}(p)-\frac{(\alpha-\beta)^2(p+1)^2}{2p^4}U_{\alpha,\beta}(0)+ \frac{\beta^2-\alpha^2}{p^2}U_{\alpha,\beta}(0)\\
+\frac{2\alpha^2}{\beta p}U_{\beta,\alpha}(p)-\frac{2\beta^2}{\alpha p}U_{\alpha,\beta}(p)
+\frac{\beta-\alpha}{p^3}(U_{\alpha,\beta}(p)
+U_{\beta,\alpha}(p))=0,
\end{multline*}

By \eqref{tt4} we obtain
\begin{multline*}
\frac{(\alpha-\beta)(p+1)^2}{p^3} U_{\alpha,\beta}(p)-\frac{(\alpha-\beta)^2(p+1)^2}{2p^4}U_{\alpha,\beta}(0)- \frac{\alpha^2-\beta^2}{p^2}U_{\alpha,\beta}(0)\\
-\frac{2\alpha^2(\alpha-\beta)}{\beta p^2}U_{\alpha,\beta}(0)+\Bigl(\frac{2\alpha^2}{\beta p}
-\frac{2\beta^2}{\alpha p}\Bigr)U_{\alpha,\beta}(p)
\\-\frac{2(\alpha-\beta)}{p^3}U_{\alpha,\beta}(p)+\frac{(\alpha-\beta)^2}{p^4}U_{\alpha,\beta}(0)=0,
\end{multline*}
Dividing by $\alpha-\beta$ we conclude that for $\alpha\not=\beta$ we have 
\begin{multline*}
\Bigl(\frac{(p+1)^2}{p^3} +\frac{2(\alpha^2+\alpha\beta+\beta^2)}{\alpha\beta p}-\frac{2}{p^3}\Bigr)U_{\alpha,\beta}(p) 
\\ =\Bigl(\frac{(\alpha-\beta)(p+1)^2}{2p^4}+\frac{\alpha+\beta}{p^2}+\frac{2\alpha^2}{\beta p^2}-\frac{\alpha-\beta}{p^4}\Bigr)U_{\alpha,\beta}(0).
\end{multline*}
Fix $\alpha=1$. Then
\begin{multline*}
\Bigl(\frac{2\beta}{p}+O(1)\Bigr)U_{\alpha,\beta}(p)  =\Bigl(\Bigl(\frac{1}{p^2}+\frac{1}{p^4}-\frac{(p+1)^2}{2p^4}\Bigr)\beta+O(1)\Bigr)U_{\alpha,\beta}(0)\\=
\Bigl(\frac{(p-1)^2}{2p^4}\beta+O(1)\Bigr)U_{\alpha,\beta}(0),\qquad \beta\to\infty.
\end{multline*}

Suppose now that $m>2$, that is, $p>1$. Then for some $C>0$ and for all $\beta\ge \beta_0$ we have 
$$
\int_{\mathbb C}  \frac{  |z|^{2p} e^{-  \beta |z|^{2p}}  }{ S_{1,2p}(|z|^2)}\,dA(z)\ge C \int_{\mathbb C}  \frac{e^{-  \beta |z|^{2p}}  }{ S_{1,2p}(|z|^2)}\,dA(z).
$$
Since
\begin{align*}
\int_{\mathbb C}  \frac{e^{-  \beta |z|^{2p}}  }{ S_{1,2p}(|z|^2)}\,dA(z)&\ge C\beta^{-1/p},\qquad \beta\to\infty,\\
\int_{\mathbb C}  \frac{|z|^{2p} e^{-  \beta |z|^{2p}}  }{ S_{1,2p}(|z|^2)}\,dA(z)&=O(\beta^{-(p+1)/p}),\qquad \beta\to\infty,
\end{align*}
we obtain a contradiction. Thus, $m=2$. 
\end{proof}

\begin{proof}[Proof of Theorem A] The implication $(4)\implies(1)$ is contained in the book of Zhu \cite[Section 3.2]{Z}.
The proof of the remaining parts follows from Lemmata \ref{mnoneven} and \ref{moments}.
\end{proof}

\end{document}